\documentclass[12pt, a4paper]{article}

\usepackage{amsmath, amssymb, amsthm}
\usepackage{enumerate}
\usepackage{mathrsfs}
\usepackage{hyperref}
\hypersetup{
                        bookmarksnumbered=true,
                        hypertexnames=true,
                        breaklinks=true,
                        linkbordercolor={1 1 1},
                        plainpages=false,
                        colorlinks=true,
                        linkcolor= blue,	
                        citecolor= blue, 
                        urlcolor= blue, 
                        pdfauthor={Sabina Pannek},
                        pdftitle={},
                        pdfsubject={},
                        pdfkeywords={},
}
\usepackage{sectsty}

\subsectionfont{\fontsize{12}{15}\selectfont}

\theoremstyle{plain}
\newtheorem{theorem}{Theorem}[section]
\newtheorem{lemma}[theorem]{Lemma}
\newtheorem{corollary}[theorem]{Corollary}
\newtheorem{proposition}[theorem]{Proposition}

\theoremstyle{definition}
\newtheorem{example}[theorem]{Example}
\newtheorem{definition}[theorem]{Definition}

\newcommand{\GL}{{\rm GL}}
\newcommand{\s}{\mathcal}

\newcommand{\ppd}{\mathrm{ppd}}
\newcommand{\lcm}{{\rm lcm}}
\newcommand{\ord}[2]{{\rm ord}\mathopen{}\left(#1; #2\right)\mathclose{}} 
\renewcommand{\mod}[1]{\,({\rm mod}\;{#1})}

\begin{document}

\title{Number of irreducible polynomials whose compositions with monic monomials have large irreducible factors}
\author{Sabina B. Pannek}
\date{}

\maketitle

\begin{abstract}
Given a prime power~$q$ and positive integers $m,t,e$ with $e > mt/2$, we determine the number of all monic irreducible polynomials $f(x)$ of degree~$m$ with coefficients in $\mathbb{F}_q$ such that $f(x^t)$ contains an irreducible factor of degree~$e$. Polynomials with these properties are important for justifying randomised algorithms for computing with matrix groups.
\end{abstract}

algorithms for matrix groups

\section{Introduction}\label{s:intro}

Throughout this paper let $q$ be a prime power, let $\mathbb{F}_q$ denote the finite field of size~$q$, and let $\mathbb{F}_q[x]$ be the ring of all polynomials with coefficients in~$\mathbb{F}_q$. The set of all positive integers is denoted by $\mathbb{N}$. 

Generalising the notion of irreducible polynomials in~$\mathbb{F}_q[x]$, we refer to $f \in \mathbb{F}_q[x]$ as \emph{$t$-hyper-irreducible} ($t \in \mathbb{N}$) if $f(x^t)$ is irreducible over~$\mathbb{F}_q$.
Thus, ``$1$-hyper-irreducible'' simply means ``irreducible''. If $f\in \mathbb{F}_q[x]$ is reducible, then $f(x^t)$ is reducible for all $t \in \mathbb{N}$. This shows that $t$-hyper-irreducible polynomials are irreducible.

\subsection{Statement of main results}

While irreducible polynomials of degree~$m$ over~$\mathbb{F}_q$ exist for every positive integer~$m$ (see \cite[Corollary~$2.11$]{MR1429394}), this is not true for $t$-hyper-irreducible polynomials. Our first theorem sheds light on when $\mathbb{F}_q[x]$ contains $t$-hyper-irreducible polynomials of a given degree. In fact it reveals even more, specifying all triples~$(m,t,e) \in \mathbb{N}^3$ with $e > mt/2$ for which $\mathbb{F}_q[x]$ contains an irreducible polynomial~$f$ such that $\deg(f) = m$ and $f(x^t)$ has an irreducible factor of degree~$e$. Such polynomials are referred to as \emph{almost $t$-hyper-irreducible}. Note that any $t$-hyper-irreducible polynomial is almost $t$-hyper-irreducible. As outlined in Subsection~$\ref{ss:motivation}$, polynomials with these properties are relevant for designing efficient algorithms for exploring properties of matrix groups.

We characterise the existence of almost $t$-hyper-irreducible polynomials using the expression~$\ord{q}{r}$, which means the smallest positive integer~$n$ such that $q^n-1$ is divisible by $r$.

\begin{theorem}\label{t:existence:1}
	Let $m,t,e \in \mathbb{N}$ satisfy $e > mt/2$.
	Then $\mathbb{F}_q[x]$ contains an irreducible polynomial~$f$ such that $\deg(f) = m$ and $f(x^t)$ has an irreducible (over~$\mathbb{F}_q$) factor of degree~$e$ if and only if 
			\begin{align}
				\gcd(t,q) =1 \text{ and } \ord{q}{(q^m-1)t} = e. \label{eq:t:existence}
			\end{align}
\end{theorem}

Hence, $t$-hyper-irreducible polynomials of degree~$m$ exist in~$\mathbb{F}_q[x]$ if and only if $t$ and $q$ are coprime and $\ord{q}{(q^m-1)t} = mt$. As shown in Corollary~$\ref{c:existence:3}(b)$, this is equivalent to $q^m-1$ being divisible by $\gcd(t,4) \prod_{i=1}^\ell t_i$, where $t_1, \dots, t_\ell$ are the distinct odd prime divisors of~$t$.

\begin{definition}\label{d:num t-hyper-irred}
	For $m,t \in \mathbb{N}$ we write $N_q^*(m,t)$ to denote the number of all monic $t$-hyper-irreducible polynomials~$f(x) \neq x$ of degree~$m$ over~$\mathbb{F}_q$. Further, we define $N_q^*(m) = N_q^*(m,1)$.
\end{definition}

A formula for $N_q^*(m)$ is known and dates back to Gau\ss{}~\cite[p.~$611$]{MR0188045} who proved it for $q$ prime even though his arguments also hold for $q$ being a prime power. We have
	\begin{align}
		N_q^*(m) = \frac{1}{m} \sum_{m_0 \mid m} \mu(m_0)(q^{m/m_0}-1),\label{eq:N_q(m)}
	\end{align}
where $\mu:\mathbb{N} \rightarrow \{-1,0,1\}$ is the \emph{Moebius function} defined by
	\begin{align}
		\mu(n) = \begin{cases}
							1, & \text{if }n=1,\\
							(-1)^k, & \text{if $n$ is the product of $k$ distinct primes},\\
							0, & \text{if $n$ is divisible by the square of a prime}.
		\end{cases}\label{eq:moebius}
	\end{align}

Our next theorem generalises Gau\ss{}' result presenting a formula for the number of all monic $t$-hyper-irreducible polynomials $f(x) \neq x$ over~$\mathbb{F}_q$ of some given degree, assuming that any exist.
We also specify a good upper and lower bound for that value.

Let $\varphi:\mathbb{N}\rightarrow \mathbb{N}$ denote Euler's totient~function.
Recall that integers $m,t \in \mathbb{N}$ satisfying $N_q^*(m,t) \neq 0$ are characterised in Theorem~$\ref{t:existence:1}$ and Corollary~$\ref{c:existence:3}(b)$.

\begin{theorem}\label{t:number1:1}
	Let $m,t \in \mathbb{N}$ be such that $N_q^*(m,t) \neq 0$.
		\begin{enumerate}[$(a)$]	
			\item Let $J = \Bigl\{j \in \mathbb{N} \mid j \text{ divides } m,\, \gcd\Bigl( \frac{q^m-1}{q^{m/j}-1}, t \Bigr) =1\Bigr\}$. Then
				\[ N_q^*(m,t) = \frac{\varphi(t)}{mt} \sum_{j \in J} \mu(j)(q^{m/j}-1).\]
			\item We have
				\[ \frac{\varphi(t) (q^m-1)}{t (m+1)} \leq N_q^*(m,t) \leq \frac{\varphi(t)(q^m-1)}{tm}. \]%
		\end{enumerate}%
\end{theorem}

Finally, we demonstrate how to deduce the number of almost $t$-hyper-irreducible polynomials which are not $t$-hyper-irreducible from the special case of $t$-hyper-irreducible polynomials covered in Theorem~$\ref{t:number1:1}$. For a natural number~$t$ and a prime~$s$, let $(t)_s$ be the \emph{$s$-part} of~$t$, that is the largest power of $s$ dividing~$t$. We write $(t)_{s'} = t/(t)_s$ and call $(t)_{s'}$ the \emph{$s'$-part} of~$t$.
Note that, if $s$ does not divide $t$, then $(t)_s = 1$ and $(t)_{s'} = t$.

\begin{theorem}\label{t:number2:1}
	Let $m,t,e \in \mathbb{N}$ satisfy $mt/2 < e < mt$ and $\eqref{eq:t:existence}$.
	Then $m\mid e$, the integer~$s = t/\gcd(e/m,t)$ is an odd prime, and the number of all monic, irreducible polynomials~$f(x) \neq x$ in $\mathbb{F}_q[x]$ such that $\deg(f) = m$ and $f(x^t)$ contains an irreducible (over~$\mathbb{F}_q$) factor of degree~$e$ is given by $N_q^*\bigl(m, (t)_{s'}\bigr)$.	
\end{theorem}

Theorems~$\ref{t:existence:1}, \ref{t:number1:1}, \ref{t:number2:1}$ are proved in Section~$\ref{s:proofs}$. 
The existence and the number of some explicit almost $t$-hyper-irreducible polynomials are discussed in Examples~$\ref{e:existence:4}$ and $\ref{e:number2:2}$.

\subsection{Motivation}\label{ss:motivation}

Consider the finite general linear group~$\GL(\s{V})$, the group of all non-singular linear mappings on a finite vector space~$\s{V}$. An element of~$\GL(\s{V})$ is called \emph{fat} if it leaves invariant and acts irreducibly on a subspace of dimension~$e> d/2$. Such elements were first defined by Niemeyer, Praeger and the author in~\cite{MR2890285}. Fat elements generalise the concept of $\ppd$-elements, which are defined by the property of having orders divisible by certain primes called \emph{p}rimitive \emph{p}rime \emph{d}ivisors. In 1997, Guralnick, Penttila, Praeger, and Saxl \cite{MR1658168} classified all subgroups of $\GL(\s{V})$ containing $\ppd$-elements. 
Their work plays an important role in computational group theory for proving results related to the generation of finite simple groups \cite{MR2422303, MR1800754} and designing randomised algorithms for working with groups of matrices over finite fields \cite{MR1625479}. There is also a wide variety of applications in other fields including number theory \cite{MR1343675}, permutation group theory \cite{MR2308171, MR3356280, MR1800730}, and geometry \cite{MR3487142, MR2994631}.

The principal motivation for the work reported in this paper is our wish to carry out an analogous classification of all subgroups of $\GL(\s{V})$ containing fat elements and, moreover, to determine the proportion of fat elements in the relevant groups. Having achieved this goal we then aim to design new randomised algorithms based on fat elements. The purpose is twofold: Firstly, testing for fatness is computationally cheaper than testing whether an element is a ppd-element, and so it is possible that dropping the ppd-property could improve various existing algorithms. Secondly, the results presented in the author's PhD thesis~\cite{mythesis} suggest that fat elements may help to recognise certain matrix groups for which there are no recognition algorithms yet.

As in the case of the $\ppd$-classification we pattern our analysis by Aschbacher's classification \cite{MR746539} of the maximal subgroups of $\GL(\s{V})$ into nine partly overlapping classes~$\mathscr{C}_1, \dots, \mathscr{C}_8$ and~$\mathscr{S}$. 
In her PhD thesis~\cite{mythesis} the author proves that the existence and number of fat elements in a group~$G$ belonging to Aschbacher's classes $\mathscr{C}_2$ or $\mathscr{C}_3$ are linked to the existence, and respectively the number, of almost $t$-hyper-irreducible polynomials. 
The occurrence of fat elements in groups belonging to Aschbacher's classes $\mathscr{C}_2, \mathscr{C}_3$, as well as further results from~\cite{mythesis}, will be covered in separate publications. Here, we only add that the $\mathscr{C}_2$-case and the $\mathscr{C}_3$-case rely (among other things) on Proposition~$\ref{p:charPoly}$ below. It shows that the composition $f(x^t)$ of a monic polynomial~$f \in \mathbb{F}_q[x]$ and $x^t$ arises as the characteristic polynomial of certain $(t \times t)$-block monomial matrices with block length $\deg(f)$.

\begin{proposition}\label{p:charPoly}
	Let $m,t \in \mathbb{N}$ and let $g_1, \dots, g_t$ be invertible $(m \times m)$-matrices over~$\mathbb{F}_q$.
	Let $f \in \mathbb{F}_q[x]$ be the characteristic polynomial of the product $g_1 \cdots g_t$.
	Then $f(x^t)$ is the characteristic polynomial of the matrix
				\[ M =\begin{pmatrix}
				\mathbf{0} & g_1 & \mathbf{0} & \cdots & \mathbf{0} \\ 
				\vdots & \mathbf{0} & g_2 & \ddots & \vdots \\
				\vdots & \vdots & \ddots & \ddots & \mathbf{0} \\
				\mathbf{0} & \vdots &  & \ddots & g_{t-1} \\
				g_t & \mathbf{0} & \cdots & \cdots & \mathbf{0}
		\end{pmatrix},\]
	where $\mathbf{0}$ denotes the $(m \times m)$-zero matrix over~$\mathbb{F}_q$.
\end{proposition}

\begin{proof}
	Let $\mathbf{1}$ be the $(m \times m)$-identity matrix over~$\mathbb{F}_q$, let $h$ be the product $g_1 \cdots g_t$, and let $B$ be the $(t \times t)$-block diagonal matrix with diagonal blocks $\mathbf{1}$, $g_1$, $g_1g_2$, \dots, $g_1\cdots g_{t-1}$. Then
	\[B M B^{-1}
		=\begin{pmatrix}
				\mathbf{0} & \mathbf{1} & \mathbf{0} & \cdots & \mathbf{0} \\ 
				\vdots & \mathbf{0} & \mathbf{1} & \ddots & \vdots \\
				\vdots & \vdots & \ddots & \ddots & \mathbf{0} \\
				\mathbf{0} & \vdots &  & \ddots & \mathbf{1} \\
				h & \mathbf{0} & \cdots & \cdots & \mathbf{0}
		\end{pmatrix}.\]
	Hence, the characteristic polynomial of~$M$ is given by the determinant of
		\[ A = \begin{pmatrix}
				x\mathbf{1} & -\mathbf{1} & \mathbf{0} & \cdots & \mathbf{0} \\ 
				\vdots & x\mathbf{1} & -\mathbf{1} & \ddots & \vdots \\
				\vdots & \vdots & \ddots & \ddots & \mathbf{0} \\
				\mathbf{0} & \vdots &  & \ddots & -\mathbf{1} \\
				-h & \mathbf{0} & \cdots & \cdots & x\mathbf{1}
		\end{pmatrix}, \]
	and it remains to verify that $\det(A) = f(x^t)$. We proceed as follows. First, using elementary transformations of rows and columns, we transform~$A$ (in $t$ steps) into a matrix $A_t$ whose determinant is equal to $\det(A)$. Then we verify that $\det(A_t) = f(x^t)$.
	
 	Set $A_0 = A$. For $i \in \{1, \dots, t-1\}$ we transform $A_{i-1}$ into the matrix $A_i$ by multiplying the $i$-th row of blocks of $A_{i-1}$ by $x$ and adding it to the $(i+1)$-th row of blocks of~$A_{i-1}$. Then
 			\[ A_{t-1} = \begin{pmatrix}
				x\mathbf{1} & -\mathbf{1} & \mathbf{0} & \cdots & \mathbf{0} \\ 
				x^2\mathbf{1} & \mathbf{0} & -\mathbf{1} & \ddots & \vdots \\
				\vdots & \vdots & \ddots & \ddots & \mathbf{0} \\
				x^{t-1}\mathbf{1} & \vdots &  & \ddots & -\mathbf{1} \\
				x^t\mathbf{1}-h & \mathbf{0} & \cdots & \cdots & \mathbf{0}
		\end{pmatrix}. \]
	Let $A_t$ be the matrix obtained from $A_{t-1}$, for each $i \in \{1, \dots, t-1\}$, by multiplying the $(i+1)$-th column of blocks of $A_{t-1}$ by $x^i$ and adding it to the first column of blocks of $A_{t-1}$. Then 
 			\[ A_t = \begin{pmatrix}
				\mathbf{0} & -\mathbf{1} & \mathbf{0} & \cdots & \mathbf{0} \\ 
				\mathbf{0} & \mathbf{0} & -\mathbf{1} & \ddots & \vdots \\
				\vdots & \vdots & \ddots & \ddots & \mathbf{0} \\
				\mathbf{0} & \vdots &  & \ddots & -\mathbf{1} \\
				x^t\mathbf{1}-h & \mathbf{0} & \cdots & \cdots & \mathbf{0}
		\end{pmatrix}. \]	
	After repeatedly applying Laplace expansion along the respective first row (for $m(t-1)$ times), we obtain $\det(A_t) =  (-1)^{(m+3)m(t-1)}\det(x^t\mathbf{1}-h) = \det(x^t\mathbf{1}-h) = f(x^t)$, as asserted.
\end{proof}

\section{Preliminaries}

The proofs of Theorems~$\ref{t:existence:1}, \ref{t:number1:1}, \ref{t:number2:1}$ rely mainly on elementary number theory and some facts about roots of polynomials over finite fields. We discuss all preliminary results in this section. For a prime~$s$, recall the notions $(t)_s, (t)_{s'}$ for the $s$-part, and respectively the $s'$-part, of a positive integer~$t$.

\subsection{The order of an integer modulo \emph{r}}

Given $r \in \mathbb{N}$, consider the ring $\mathbb{Z}/r\mathbb{Z}$ of integers modulo~$r$ and its group of units $(\mathbb{Z}/r\mathbb{Z})^*$. Elements of $(\mathbb{Z}/r\mathbb{Z})^*$ are of the form $a+r\mathbb{Z}$, where $a$ is a positive integer coprime to $r$. In particular, $a \neq r$ unless $r = 1$, in which case $\mathbb{Z}/1\mathbb{Z}$ is the zero ring and $(\mathbb{Z}/1\mathbb{Z})^*$ is the trivial group. We write $\ord{a}{r} = m$ to denote that the element $a +r\mathbb{Z} \in (\mathbb{Z}/r\mathbb{Z})^*$ has order~$m$. Equivalently, $m$ is the smallest positive integer such that $a^m-1$ is divisible by~$r$. In fact, we have $\ord{a}{r} = m$ if and only if $r$ divides $a^m-1$ but $r$ does not divide $a^i-1$ for any proper divisor~$i$ of~$m$.

Recall that we use the letter $\varphi$ to denote Euler's totient function, that is $\varphi: \mathbb{N} \rightarrow \mathbb{N}$, $r \mapsto \vert (\mathbb{Z}/r\mathbb{Z})^* \vert$. Further, we let $\lambda: \mathbb{N} \rightarrow \mathbb{N}$, $r \mapsto \exp( (\mathbb{Z}/r\mathbb{Z})^*)$ be the Carmichael function assigning to each positive integer $r$ the exponent of the group $(\mathbb{Z}/r\mathbb{Z})^*$, that is the least common multiple of the orders of all elements in $(\mathbb{Z}/r\mathbb{Z})^*$. (The Carmichael function was first introduced by Carmichael \cite{MR1558896} in 1910.) If $\gcd(a,r) = 1$ then, by definition,
	\begin{align}
		 \ord{a}{r} \mid \lambda(r) \mid \varphi(r). \label{eq:ord|lambda|phi}
	\end{align}
Several other basic properties of $\ord{a}{r}$ are listed in our first lemma below.

\begin{lemma}\label{l:ord:1}
	Let $a, r \in \mathbb{N}$ be coprime.
	\begin{enumerate}[$(a)$]
		\item Let $k \in \mathbb{N}$. Then $r \mid a^k-1$ if and only if $\ord{a}{r} \mid k$.
		\item If $r'\mid r$, then $\ord{a}{r'} \mid \ord{a}{r}$.
		\item Let $s \in \mathbb{N}$ be coprime to $ar$. Then $\ord{a}{rs} = \lcm \{ \ord{a}{r}, \ord{a}{s}\}$.
		\item Suppose that $r$ is prime and let $k  \geq 2$ be an integer. Then 
 			\[\ord{a}{r^k} = \ord{a}{r^{k-1}}r^j, \quad \text{for some } j \in \{0,1\}.\]
 		\item If $r$ is a prime and $k \in \mathbb{N}$, then $\ord{a}{r^k} \mid (r-1)r^{k-1}$.\\
 		Further, $\ord{a}{2^k} \mid 2^{k-2}$ for~$k \geq 3$.
		\item Let $t \in \mathbb{N}$ be coprime to $a$. Then $\ord{a}{rt} \leq \ord{a}{r} t$.
	\end{enumerate}
\end{lemma}%

\begin{proof}
	\begin{enumerate}[$(a)$]
	\item Let $m = \ord{a}{r}$ and let $\ell, s$ be non-negative integers such that $s < m$ and $k = \ell m + s$.
		Since $a^m \equiv 1\mod{r}$ we have $a^{\ell m} \equiv 1\mod{r}$. Then $a^k \equiv a^s\mod{r}$. Thus, $r \mid a^k-1$ if and only if $a^s \equiv 1\mod{r}$, which (recalling that $s < m$) is the case if and only if $s = 0$, that is (recalling that $k = \ell m + s$) if and only if $m \mid k$.
		
	\item Let $m = \ord{a}{r}$. Since $r \mid a^m-1$, any divisor $r'$ of $r$ also divides $a^m-1$. Then part~$(a)$ of the current lemma 
	yields $\ord{a}{r'} \mid  m$.
		
	\item Since $(\mathbb{Z}/rs\mathbb{Z})^* \cong (\mathbb{Z}/r\mathbb{Z})^* \times (\mathbb{Z}/s\mathbb{Z})^*$, any element $a + rs\mathbb{Z} \in (\mathbb{Z}/rs\mathbb{Z})^*$ has order equal to the least common multiple of the orders of $a + r\mathbb{Z} \in (\mathbb{Z}/r\mathbb{Z})^*$ and $a + s\mathbb{Z} \in (\mathbb{Z}/s\mathbb{Z})^*$.
		
	\item Let $m = \ord{a}{r^{k-1}}$. By part~$(b)$ of the current lemma, $m \mid \ord{a}{r^k}$. It remains to show that $\ord{a}{r^k} \mid mr$. By the definition of~$m$ we have $r^{k-1} \mid a^m-1$, and in particular, $a^m \equiv 1\mod{r}$. Thus, $r$ divides $a^{m(r-1)}+ \dots + a^m +1 = (a^{mr}-1)/(a^m-1)$, and equivalently $r(a^m-1) \mid a^{mr}-1$. Recalling that $r^{k-1} \mid a^m-1$, we get $r^k \mid a^{mr}-1$. Then $\ord{a}{r^k} \mid mr$ by part~$(a)$ of the current lemma.
	
	\item The assertion follows directly from~$\eqref{eq:ord|lambda|phi}$ and the formula to calculate $\lambda(r)$ given in \cite[p.~$232$]{MR1558896} (see \cite[p.~$29$]{MR2028675} for a more recent reference).
	
	\item Let $t'$ be the largest divisor of~$t$ which is coprime to~$r$. We then have $\gcd(rt/t', t')= 1$. Thus, by part~$(c)$ of the current lemma, $\ord{a}{rt}$ equals the least common multiple of, and thus divides the product of, $\ord{a}{rt/t'}$ and $\ord{a}{t'}$. Using $\ord{a}{t'} \leq t'$ we obtain
		\[ \ord{a}{rt} \leq \ord{a}{\frac{rt}{t'}} t'. \]%
	We prove the assertion by showing that $\ord{a}{rt/t'} \leq \ord{a}{r}t/t'$. To this end, let $s_1, \dots, s_\ell$ be the distinct prime divisors of~$r$. Recall (from the definition of~$t'$) that each prime factor of $t/t'$ divides~$r$. Hence, $s_1, \dots, s_\ell$ are also (all of) the distinct prime divisors of~$rt/t'$. According to part~$(c)$ of this lemma we have
		\[\ord{a}{\frac{rt}{t'}} = \lcm\left\{ \ord{a}{\Bigl(\frac{rt}{t'}\Bigr)_{s_1}}, \dots, \ord{a}{\Bigl(\frac{rt}{t'}\Bigr)_{s_\ell} } \right\}.\]
	For $i \in \{1, \dots, \ell\}$, repeatedly applying part~$(d)$ of the current lemma shows that $\ord{a}{ (rt/t')_{s_i} }$ divides $ \ord{a}{(r)_{s_i} } (t/t')_{s_i}$. Hence,
		\[  \ord{a}{\frac{rt}{t'}} \mid \lcm\bigl\{ \ord{a}{ (r)_{s_1} }, \dots, \ord{a}{ (r)_{s_\ell}} \bigr\} \frac{t}{t'}. \]%
	Using part~$(c)$ of the current lemma it follows that $\ord{ a}{rt/t'}$ divides, and thus is less than or equal to, $\ord{a}{r} t/t'$. \qedhere
	\end{enumerate}
\end{proof}%

If $a,r,t \in \mathbb{N}$ are such that $\gcd(a,rt) =1$ then by Lemma~$\ref{l:ord:1}(f)$ we have $\ord{a}{rt} \leq \ord{a}{r}t$. We are particularly interested in the situation where $\ord{a}{rt}$ is large in the sense that $\ord{a}{rt} > \ord{a}{r}t/2$. The case $a = 1$ is trivial. (If $a = 1$ then $\ord{1}{rt} > \ord{1}{r}t/2$ precisely when~$t = 1$.)
So, in what follows, we assume that $a \geq 2$.

\begin{lemma}\label{l:ord:2}
	Let $a,r,t \in \mathbb{N}$ be such that $a \geq 2$ and $\gcd(a,rt) = 1$. Let $m = \ord{a}{r}$ and assume that $\ord{a}{rt} >mt/2$. Then the following hold.
		\begin{enumerate}[$(a)$]
			\item We have $\gcd(t, (a^m-1)/r) = 1$ and $\gcd(4,t) \mid r$.
			\item If $t$ has a prime divisor~$s$ which does not divide $r$ then $s$ is uniquely determined and $s \neq 2$.
		\end{enumerate}
\end{lemma}%

\begin{proof}
	\begin{enumerate}[$(a)$] 
		\item Seeking a contradiction, suppose that there is a prime divisor~$s$ of~$t$ which divides $(a^m-1)/r$. Then $rs \mid a^m-1$. Hence, by Lemma~$\ref{l:ord:1}(a)$ we get $\ord{a}{rs} \mid m$. Since $\ord{a}{r} = m$, Lemma~$\ref{l:ord:1}(b)$ reveals that $m \mid \ord{a}{rs}$. Thus, $\ord{a}{rs} = m$.
		Then Lemma~$\ref{l:ord:1}(f)$ yields the contradiction $\ord{a}{rt} \leq \ord{a}{rs}t/s \leq mt/2$.
 
 		Next, again seeking a contradiction, suppose that $\gcd(4,t) \nmid r$. According to Lemma~$\ref{l:ord:1}(f)$ we have $\ord{a}{rt} \leq \ord{a}{r(t)_2} \, (t)_{2'}$ which by Lemma~$\ref{l:ord:1}(c)$ is equivalent to
 			\begin{align}
 				\ord{a}{rt} \leq \lcm\biggl\{ \underbrace{\ord{a}{ \frac{r}{\gcd(2,r)}}}_{=:\alpha }, \underbrace{\ord{a}{\gcd(2,r)(t)_2}}_{=:\beta}\biggr\}\, (t)_{2'}.\label{eq:ord:2:1}
 			\end{align}
 		Since $\gcd(4,t)$ does not divide $r$, we get $2 \mid \gcd(2,r)(t)_2$ if $\gcd(2,r) = 1$, and $8 \mid\gcd(2,r)(t)_2$ if $\gcd(2,r) =2$. By Lemma~$\ref{l:ord:1}(e)$, it follows that 
 			\begin{align}
 				\beta \mid \frac{(t)_2}{2}.\label{eq:ord:2:2}
 			\end{align}
 		Now, by Lemma~$\ref{l:ord:1}(b)$, $\alpha$ divides $m$. Combining this fact with $\eqref{eq:ord:2:1}$ and $\eqref{eq:ord:2:2}$ yields the contradiction $\ord{a}{rt} \leq \lcm\{ m, (t)_2/2\} \, (t)_{2'} \leq mt/2$.
 		
 		\item Suppose that $t$ has two prime divisors~$s$ and $\ell$ which do not divide~$r$. By part~$(a)$ of the current lemma we have $2 \nmid s\ell$.
 		Seeking a contradiction, assume that $s \neq \ell$. Then by Lemma~$\ref{l:ord:1}(c)(f)$ we get 
 			\begin{align}
 				\ord{a}{rt} \leq \lcm\bigl\{\ord{a}{r},\, \ord{a}{(t)_s}, \,\ord{a}{(t)_\ell } \bigr\} \frac{ t}{ (t)_s (t)_\ell }.\label{eq:ord:2:3}
 			\end{align}
 		Now, by Lemma~$\ref{l:ord:1}(e)$, $\ord{a}{(t)_s}$ divides $(s-1)(t)_s/s$, and $\ord{a}{(t)_\ell}$ is a divisor of $(\ell-1)(t)_\ell/\ell$. Moreover, $\ord{a}{r} = m$ by definition. Hence, $\eqref{eq:ord:2:3}$ yields
 			\[\ord{a}{rt} \leq \lcm\left\{m, (s-1)\frac{ (t)_s}{s}, (\ell-1) \frac{ (t)_\ell}{\ell}  \right\} \frac{ t}{ (t)_s (t)_\ell},\]
 		and thus
 			\begin{align}
 				\ord{a}{rt} \leq \lcm\{ s-1, \ell-1 \} \frac{m t}{ s \ell}.\label{eq:ord:2:4}
 			\end{align}
 		Recall that $s$ and $\ell$ are odd. Then $\lcm\{ s-1, \ell-1\}  < s \ell/2$, whence $\ord{a}{rt} < mt/2$ by $\eqref{eq:ord:2:4}$. As this is not true, we conclude that $s = \ell$. \qedhere
	\end{enumerate}%
\end{proof}%

We next classify all triples~$(a,r,t) \in \mathbb{N}^3$ satisfying $a \geq 2$, $\gcd(a,rt) = 1$ and $\ord{a}{rt} = \ord{a}{r}t$. The implication ``$(b) \Rightarrow (a)$'' is essentially proved in~\cite[Theorem~$3.34$]{MR1429394}.%

\begin{proposition}\label{p:ord:3}
	Let $a,r,t \in \mathbb{N}$ be such that $a \geq 2$ and $\gcd(a,r) = 1$. Let $m = \ord{a}{r}$. The following are equivalent.
	\begin{enumerate}[$(a)$]
		\item We have 	$\gcd(a,rt) = 1$ and $\ord{a}{rt} = mt$.
		\item Every prime divisor of~$t$ divides $r$ but not $(a^m-1)/r$, and $\gcd(4,t) \mid r$.
	\end{enumerate}\vskip-.3cm
\end{proposition}%

\begin{proof}
	If $t =1$ then there is nothing to show. (In this case condition~$(a)$ simplifies to $\gcd(a,r) = 1$, $\ord{a}{r} = m$, and both equations hold by assumption, while condition~$(b)$ is trivially true.) If $r = 1$ then $m = 1$ and condition~$(a)$ simplifies to $\gcd(a,t) = 1$, $\ord{a}{t} = t$, which is true if and only if $t = 1$, as asserted. We may thus assume that $r,t \geq 2$.
		
	First, suppose that condition~$(b)$ holds. Because each prime divisor of~$t$ divides~$r$, recalling that $\gcd(a,r) = 1$, we see that $\gcd(a,rt) = 1$. Further, by \cite[Theorem~$3.34$]{MR1429394} we have $\ord{a}{rt} = mt$.
	(In order to see that we may indeed apply \cite[Theorem~$3.34$]{MR1429394}, observe that $4 \mid t$ implies $4 \mid r \mid a^m-1$. Hence, $t \equiv 0\mod{4}$ implies that $a^m \equiv 1 \mod{4}$.)
	
	Conversely, suppose that condition~$(a)$ holds. From Lemma~$\ref{l:ord:2}(a)$ we know that
		\[ \gcd\biggl(t, \frac{a^m-1}{r}\biggr) = 1 \quad \text{and} \quad \gcd(4,t) \mid r. \]
	It remains to show that every prime divisor of~$t$ divides~$r$. Seeking a contradiction, assume that some prime $s$ divides $t$ and $s \nmid r$. By Lemma~$\ref{l:ord:1}(c)$ we have $\ord{a}{rt} = \lcm\bigl\{ \ord{a}{r (t)_{s'}}, \ord{a}{(t)_s}\bigr\}$, whence
		\begin{align}
			\ord{a}{rt} \leq \ord{a}{r (t)_{s'}} \, \ord{a}{(t)_s}.\label{eq:ord:3:1}
		\end{align}
	Note that by Lemma~$\ref{l:ord:1}(f)$ we have
		\begin{align}
			\ord{a}{r (t)_{s'}} \leq  \ord{a}{r} \,(t)_{s'}.\label{eq:ord:3:2}
		\end{align}	
	 Since $(t)_s > 1$, according to Lemma~$\ref{l:ord:1}(e)$ we get $\ord{a}{(t)_s} \mid (s-1)(t)_s/s$ and thus $\ord{a}{(t)_s} < (t)_s$. Combining the latter with $\eqref{eq:ord:3:1}$ and~$\eqref{eq:ord:3:2}$ yields the contradiction~$\ord{a}{rt} < \ord{a}{r}t$.%
\end{proof}%

Recall from Lemma~$\ref{l:ord:1}(b)$ that, given $a,r,t \in \mathbb{N}$ such that $a$ and $rt$ are coprime, we have $\ord{a}{r} \mid \ord{a}{rt}$, that is $\ord{a}{rt} / \ord{a}{r} \in \mathbb{N}$.%

\begin{lemma}\label{l:ord:4}
	Let $a, r, t\in \mathbb{N}$ be such that $a\geq 2$ and $\gcd(a,rt) = 1$. Let $m = \ord{a}{r}$ and suppose that $mt/2<\ord{a}{rt} < mt$. Let
		\[s = t / \gcd\biggl( \frac{\ord{a}{rt}}{m}, t \biggr).\]
	Then $s$ is an odd prime divisor of~$t$, $s\nmid a^m-1$, and $\ord{a}{r(t)_{s'}} = m(t)_{s'}$.%
\end{lemma}%

\begin{proof}
	Let $e = \ord{a}{r t}$. Since $e > mt/2$, from Lemma~$\ref{l:ord:2}(a)$ we know that
		\[\gcd \Bigl(t, \frac{a^m-1}{r}\Bigr) = 1 \quad \text{and} \quad \gcd(4,t) \mid r.\] 
		If all prime divisors of~$t$ divide~$r$, then Proposition~$\ref{p:ord:3}$ yields the contradiction $e = mt$. 
		Hence there exists a prime divisor $\ell$ of $t$ not dividing~$r$.
		By Lemma~$\ref{l:ord:2}(b)$, $\ell$ is the unique prime divisor of~$t$ which does not divide~$r$.
		Recalling that $\gcd\bigl(t, (a^m-1)/r\bigr) = 1$, it follows that
			\begin{align}
				\ell \nmid a^m-1\label{eq:ord:4:1}
			\end{align}		
		and that
			\begin{align}
				\text{every prime divisor of $(t)_{\ell'}$ divides~$r$ but not } \frac{a^m-1}{r}.\label{eq:ord:4:2}
			\end{align}
		Further, by Lemma~$\ref{l:ord:2}(b)$,
			\begin{align}
				 \ell \neq 2.\label{eq:ord:4:3}
			\end{align}
		Now, since $\gcd(4,t)$ divides $r$ (and thus $\gcd(4, (t)_{\ell'}) \mid r$), recalling that $\eqref{eq:ord:4:2}$ holds, Proposition~$\ref{p:ord:3}$ yields
		\begin{align}
			\ord{a}{r (t)_{\ell'} } = m(t)_{\ell'}.\label{eq:ord:4:4}
		\end{align}
	By Lemma~$\ref{l:ord:1}(e)$ we have
		\begin{align}
			\ord{a}{(t)_\ell} \mid  \frac{(\ell-1) (t)_\ell}{\ell}.\label{eq:ord:4:5}
		\end{align}
	Combining $\eqref{eq:ord:4:4}$, $\eqref{eq:ord:4:5}$ with Lemma~$\ref{l:ord:1}(c)$ (according to which $e$ is equal to the least common multiple, and thus divides the product, of $\ord{a}{r (t)_{\ell'} }$ and $\ord{a}{(t)_\ell}$), reveals that $e \mid  mt(\ell-1)/\ell$.
	Since (by assumption) the integer $e$ is strictly bigger than $mt/2$, it follows that $e = mt(\ell-1)/\ell.$ Then
		\[s = \frac{t}{\gcd\bigl( \frac{t(\ell-1)}{\ell}, t \bigr)} = \frac{t}{\frac{t}{\ell}\underbrace{\gcd(\ell-1, \ell )}_{=1}} = \ell,\]
	and the assertion holds by $\eqref{eq:ord:4:1}$, $\eqref{eq:ord:4:3}$, and $\eqref{eq:ord:4:4}$.
	\end{proof}%

\subsection{Roots of (irreducible) polynomials over finite fields}

The set of all non-zero elements in~$\mathbb{F}_q$ forms a cyclic group under multiplication, and we denote this group by~$\mathbb{F}_q^*$. The \emph{order} of a non-zero element~$\alpha \in \mathbb{F}_q$, written as~$\vert \alpha \vert$, refers to the order of~$\alpha$ in the cyclic group~$\mathbb{F}_q^*$. 
By saying \emph{root} of~$f \in \mathbb{F}_q[x]$ we mean an element $\omega$ in some possibly non-proper extension field of~$\mathbb{F}_q$ satisfying $f(\omega) = 0$. The \emph{splitting field} of~$f \in \mathbb{F}_q[x]$ is the smallest (with respect to inclusion) extension field of~$\mathbb{F}_q$ which contains all roots of~$f$.
We shall be using the following well-known properties of roots of irreducible polynomials. (See for example \cite{MR1429394} for a reference.)

\begin{lemma}\label{l:poly:1}
	Let $m \in \mathbb{N}$, and let $f \in \mathbb{F}_q[x]$ be irreducible of degree~$m$.
	\begin{enumerate}[$(a)$]
		\item The polynomial $f$ has $m$ distinct roots.
		\item If $\omega$ is a root of~$f$, then $\omega, \omega^q, \omega^{q^2}, \dots, \omega^{q^{m-1}}$ are all the roots of~$f$.
		\item The splitting field of~$f$ over~$\mathbb{F}_q$ is given by~$\mathbb{F}_{q^m}$.
		\item If $f \neq x$, then all roots of~$f$ lie in~$\mathbb{F}_{q^m}^*$ and have the same order.
		\item If $e \in \mathbb{F}_q[x]$ is irreducible and some root of~$e$ is a root of~$f$, then $e = \alpha f$ for some $\alpha \in \mathbb{F}_q^*$.%
	\end{enumerate}
\end{lemma}%

Given $m \in \mathbb{N}$, the subfields of~$\mathbb{F}_{q^m}$ which contain~$\mathbb{F}_q$ are precisely the fields~$\mathbb{F}_{q^n}$ with $n \mid m$. 
Since $\mathbb{F}_{q^n}^*$ is cyclic of order $q^n-1$, it follows that an element $\omega \in \mathbb{F}_{q^m}^*$  lies in a proper subfield of~$\mathbb{F}_{q^m}$ containing $\mathbb{F}_q$ if and only if $\vert \omega\vert$ divides $q^n-1$ for some proper divisor $n$ of $m$. Recalling the notion of $\ord{a}{r}$, we obtain the following.

\begin{lemma}\label{l:poly:2}
	Let $m \in \mathbb{N}$. An element $\omega \in \mathbb{F}_{q^m}^*$ does not lie in any proper subfield of~$\mathbb{F}_{q^m}$ containing $\mathbb{F}_q$ if and only if $\ord{q}{\vert \omega \vert} = m$.%
\end{lemma}

Whether or not a polynomial is irreducible can be read from the order of any of its non-zero roots.

\begin{lemma}\label{l:poly:3}
	Let $f\in \mathbb{F}_q[x]$ contain a non-zero root~$\omega$. Then $f$ is irreducible if and only if $\ord{q}{\vert \omega \vert} = \deg(f)$.
\end{lemma}

\begin{proof}	
	Assume that $f$ is irreducible. Let $m = \deg(f)$. By Lemma~$\ref{l:poly:1}(c)$, $\omega \in \mathbb{F}_{q^m}$. If $\omega$ lies in a proper subfield~$\mathbb{K}$ of~$\mathbb{F}_{q^m}$ containing~$\mathbb{F}_q$, then by  Lemma~$\ref{l:poly:1}(b)$ all roots of~$f$ lie in~$\mathbb{K}$, in which case $f$ splits over~$\mathbb{K}$. As this contradicts Lemma~$\ref{l:poly:1}(c)$, Lemma~$\ref{l:poly:2}$ yields $\ord{q}{\vert \omega \vert} = m$.
	
	Conversely, suppose that $\ord{q}{\vert \omega \vert} = \deg(f)$. Let $f_0 \in \mathbb{F}_q[x]$ be an irreducible factor of~$f$ with $f_0(\omega) = 0$. By the first part of this proof we get $\deg(f_0) = \ord{q}{\vert \omega \vert}$. Hence, $\deg(f_0) = \deg(f)$ and $f$ is irreducible.
\end{proof}

Recall from Definition~$\ref{d:num t-hyper-irred}$ that $N_q^*(m)$ is the number of all monic, irreducible polynomials $f \neq x$ of degree~$m$ in~$\mathbb{F}_q[x]$. The precise value of~$N_q^*(m)$ can be calculated via the formula~$\eqref{eq:N_q(m)}$. We give a (good) lower and an upper bound for~$N_q^*(m)$ in Lemma~$\ref{l:poly:5}(b)$ below. Proving the lower bound for~$N_q^*(m)$ involves the following estimate.

\begin{lemma}\label{l:poly:4}
	Let $a,m \in \mathbb{N}$ be such that $a,m\geq 2$ and $(m,a) \notin \{ (2,2), (4,2)$, $(6,2) \}$. Then $\sum_{n\mid m, n < m} (a^n-1) < (a^m-1)/(m+1)$.
\end{lemma}

\begin{proof}
	For a rational number $r$, let $\lceil r \rceil$ denote the smallest integer which is at least $r$, and let $\lfloor r \rfloor$ be the largest integer not greater than $r$.	
	
	First, suppose that $2 \leq m \leq 8$. Then the inequality $\sum_{n\mid m, n < m} (a^n-1) < (a^m-1)/(m+1)$ is equivalent to 
	\[ \begin{cases}
				a^2 -3a +2 > 0, & \text{if } m = 2, \\
				a^3-4a+3 > 0, & \text{if } m= 3, \\
				a^4 - 5a^2 -5a +9>0, & \text{if } m = 4, \\
				a^5 - 6a +5 > 0, & \text{if } m= 5,\\
				a^6 -7a^3 -7a^2-7a +20 >0, & \text{if } m = 6,\\
				a^7 -8a+7 > 0, & \text{if } m=7,\\
				a^8 - 9a^4-9a^2-9a+26 >0, &\text{if } m=8.
			\end{cases}\]
	Recalling that $a \geq 3$ if $m \in \{2,4,6\}$, one can easily verify that the assertion is true. So suppose that $m\geq 9$. Then $a^{\lceil m/2\rceil-1} > m+1$, and hence
		\[ a^{\lfloor m/2\rfloor +1}-1 < \underbrace{\frac{a^{\lceil m/2 \rceil-1}}{m+1}}_{> 1} a^{\lfloor m/2\rfloor +1} -1 = \frac{a^m}{m+1}-1 < \frac{a^m-1}{m+1}.\]	
	Then (using $a-1 \geq 1$) we have $(a^{\lfloor m/2\rfloor +1}-1)/(a-1) < (a^m-1)/(m+1)$, which is the same as saying that
		\begin{align}
			\sum_{i=0}^{\lfloor m/2\rfloor} a^i < \frac{a^m-1}{m+1}.\label{eq:poly:4:1}
		\end{align}
	Now, a proper divisor of~$m$ is at most equal to $\lfloor m/2\rfloor$. It follows that, $\sum_{n\mid m, n < m} (a^n-1) \leq \sum_{i=1}^{\lfloor m/2 \rfloor} (a^i-1) < \sum_{i=0}^{\lfloor m/2\rfloor} a^i$, which combined with~$\eqref{eq:poly:4:1}$ yields $\sum_{n\mid m, n < m} (a^n-1) < (a^m-1)(m+1)$, as needed.
\end{proof}

\begin{lemma}\label{l:poly:5}
	Let $m \in \mathbb{N}$. Then the following hold.
	\begin{enumerate}[$(a)$]
		\item The number of elements in~$\mathbb{F}_{q^m}^*$ which do not lie in any proper subfield of~$\mathbb{F}_{q^m}$ containing~$\mathbb{F}_q$ is given by $mN_q^*(m)$.
		\item We have $N_q^*(1) = q-1$. If $m \geq 2$ then
				\[\frac{q^m-1}{m+1} \leq N_q^*(m)  < \frac{q^m-1}{m}.\]
	\end{enumerate}
\end{lemma}%

\begin{proof}
	\begin{enumerate}[$(a)$]
	\item Let $\widehat{f} \in \mathbb{F}_q[x]$ be the product of all monic, irreducible polynomials~$f \neq x$ of degree~$m$ over~$\mathbb{F}_q$, and let $R$ be the set of all roots of~$\widehat{f}$.
	By Lemma~$\ref{l:poly:1}(a)(e)$ each irreducible factor of~$\widehat{f}$ has~$m$ distinct roots, and no two distinct irreducible factors of~$\widehat{f}$ have a root in common. Hence,
		\[ \vert R \vert = m N_q^*(m). \]
	By Lemma~$\ref{l:poly:1}(d)$, $R$ is a subset of~$\mathbb{F}_{q^m}^*$.
	If an irreducible factor~$f$~of~$\widehat{f}$ has a root in a proper subfield~$\mathbb{F}_{q^n}$ of~$\mathbb{F}_{q^m}$ then Lemma~$\ref{l:poly:1}(b)$ implies that all roots of~$f$ lie in~$\mathbb{F}_{q^n}$, which contradicts~$\mathbb{F}_{q^m}$ being the splitting field of $f$ over~$\mathbb{F}_q$ (see Lemma~$\ref{l:poly:1}(c)$). Hence,
		\[ R \subseteq \{ \alpha \in \mathbb{F}_{q^m}^* \mid \alpha \notin \mathbb{F}_{q^n} \text{ for all proper divisors $n$ of $m$} \}. \]
	Consider an element $\omega \in \mathbb{F}_{q^m}^*$ which does not lie in any proper subfield of~$\mathbb{F}_{q^m}$ containing~$\mathbb{F}_q$.
	By Lemma~$\ref{l:poly:2}$ we have $\ord{q}{\vert \omega \vert} = m$. Let $f \in \mathbb{F}_q[x]$ be the minimal polynomial of~$\omega$ over~$\mathbb{F}_q$. Since $f$ is irreducible and $f(\omega) = 0$, Lemma~$\ref{l:poly:3}$ yields $\deg(f) = m$. It follows that $\omega \in R$, whence	
		\[ R \supseteq \{ \alpha \in \mathbb{F}_{q^m}^* \mid \alpha \notin \mathbb{F}_{q^n} \text{ for all proper divisors $n$ of $m$} \}, \]
	which proves the assertion.
						
	\item By $\eqref{eq:N_q(m)}$ we have $N_q^*(1) = q-1$. Let $m\geq 2$. As we may deduce from part~$(a)$ of the current lemma 
	we have $\vert \mathbb{F}_{q^m}^*\vert = \sum_{n\mid m} nN_q^*(n)$, whence $q^m-1 \geq N_q^*(1) + mN_q^*(m)$. Recalling that $ N_q^*(1) = q-1 \geq 1$, we get $q^m -1 \geq 1 + mN_q^*(m)$, which proves the upper bound for $N_q^*(m)$.
	In order to verify the lower bound for $N_q^*(m)$, observe that by $\eqref{eq:N_q(m)}$, $(m+1) N_q^*(m) $ is equal to
		\[ \begin{cases} 
				\frac{3}{2}\bigl((2^2-1) \!-\!(2-1) \bigr) = 2^2-1, & \!\!\text{if } (m,q) = (2,2),\\
				\frac{5}{4}\bigl((2^4-1)\! -\! (2^2-1) + 0 (2^1-1)\bigr) =2^4-1, & \!\!\text{if } (m,q) = (4,2),\\
				\frac{7}{6} \bigl( (2^6-1) \!-\!(2^3-1) \!-\! (2^2-1) \!+\! (2-1) \bigr) = 2^6-1, & \!\!\text{if } (m,q) = (6,2).
			\end{cases} \]
	Hence, $N_q^*(m) = (q^m-1)/(m+1)$ if $(m,q) \in \{(2,2), (4,2),(6,2) \}$.
	
	Now, assume that  $(m,q) \notin \{(2,2), (4,2),(6,2)\}$. From part~$(a)$ of the current lemma we deduce that $mN_q^*(m) \geq \vert \mathbb{F}_{q^m}^*\vert - \sum \vert \mathbb{F}_{q^n}^* \vert $, where the sum is over all proper divisors~$n$ of~$m$. Then
		\begin{align*}
			mN_q^*(m) \geq (q^m-1) - \sum_{\substack{n\mid m, \\n <m}}  (q^n-1)
		\end{align*}
	which, using the inequality $\sum_{n\mid m,n <m} (q^n-1) < (q^m-1)/(m+1)$ given in Lemma~$\ref{l:poly:4}$, reveals that $mN_q^*(m) > (q^m-1)m/(m+1)$. Then $N_q^*(m) > (q^m-1)/(m+1)$, and the proof is complete. \qedhere
	\end{enumerate}%
\end{proof}%

\begin{lemma}\label{l:poly:6}
	Let $f \in \mathbb{F}_q[x]$ contain a non-zero root~$\omega$ and let $t \in \mathbb{N}$ be coprime to~$q$.
	Then $f(x^t)$ has a root of order $\vert \omega \vert t$.
\end{lemma}

\begin{proof}
	Since $x-\omega$ divides~$f$, the polynomial $x^t-\omega$ is a factor of~$f(x^t)$. We prove the assertion by showing that $x^t-\omega$ has a root of order~$\vert \omega \vert t$.
	\begin{enumerate}[$(i)$]
		\item We begin with the special case where $t$ is prime. Let~$\alpha$ be a root of~$x^t-~\omega$. (Hence, $\alpha^t = \omega$.) If $\vert \alpha \vert$ is divisible by~$t$, then $\vert \omega \vert t = \vert \alpha^t \vert \gcd(\vert \alpha\vert, t) = \vert \alpha \vert$, as needed.
		So, suppose that $t$ does not divide $\vert \alpha \vert$. By \cite[Theorem~$2.42(i)$]{MR1429394} there exists a $t$-th root of unity over~$\mathbb{F}_q$ of order~$t$, say~$\beta$.
		Then $(\alpha\beta)^t = \omega$, that is $\alpha\beta$ is a root of $x^t-\omega$. Moreover, since $\gcd(\vert \alpha \vert, t) = 1$ and $t = \vert\beta\vert$, the order of $\alpha\beta$ is divisible by~$t$.
	By what we have already proved, $\vert \omega \vert t = \vert \alpha\beta \vert$, as asserted.
		
		\item Let $t_1, \dots, t_\ell$ be (not necessarily distinct) primes such that $t = \prod_{i=1}^\ell t_i$.
		By part~$(i)$ of the current proof, the polynomial $x^{t_1}-\omega$ contains a root of order~$\vert \omega \vert t_1$, say~$\omega_1$.
		Applying~$(i)$ to the polynomial $x^{t_2}-\omega_1$ we see that $x^{t_2}- \omega_1$ contains a root of order $\vert \omega_1\vert t_2 = \vert \omega \vert t_1 t_2$.
		Since $x-\omega_1 \mid x^{t_1}-\omega$ we have $x^{t_2}-\omega_1 \mid x^{t_1t_2}-\omega$. It follows that $x^{t_1t_2}-\omega$ has a root of order~$\vert \omega \vert t_1t_2$.
		Repeatedly applying this procedure, we conclude that $x^t -\omega = x^{t_1 \cdots t_\ell}-\omega$ contains a root of order $\vert \omega \vert t_1 \cdots t_\ell = \vert \omega \vert t$.\qedhere
	\end{enumerate}%
\end{proof}

\subsection{Some more preliminaries}

	We conclude this section with a few more straightforward yet helpful lemmas. We start with a well-known result which we prove using Lemma~$\ref{l:ord:1}(a)$.

\begin{lemma}\label{l:more:1}
	Let $a,b,c \in \mathbb{N}$. Then  $\gcd(a^b-1, a^c-1) = a^{\gcd(b,c)}-1$.%
\end{lemma}	

\begin{proof}
	Let $\ell = \gcd(a^b-1, a^c-1)$ and $k = \gcd(b,c)$.
	For $i \in \{b,c\}$ we have $a^i-1 = (a^k-1) (a^{i - k} + a^{i -2k} + \dots + a^k + 1)$ and, in particular, $a^k-1 \mid a^i-1$. Thus, $a^k-1 \mid \ell$.
	Conversely, since for $i \in \{b,c\}$ the integer $\ell$ divides $\mid a^i-1$, Lemma~$\ref{l:ord:1}(a)$ yields $\ord{a}{\ell} \mid i$. Thus, $\ord{a}{\ell} \mid k$, and then, applying Lemma~$\ref{l:ord:1}(a)$ one more time, we conclude that $\ell \mid a^k-1$.
\end{proof}

Recall that $\varphi: \mathbb{N} \rightarrow \mathbb{N}$ denotes Euler's totient function.

\begin{lemma}\label{l:more:2}
	Let $a, b \in \mathbb{N}$. The set $\{1, \dots, ab\}$ contains $a\varphi(b)$ elements which are coprime to $b$.
\end{lemma}%

\begin{proof}
Observe that the assertion holds for $b = 1$. We thus assume that $b \geq 2$. An element $\ell \in \{1, \dots, ab\}$ is coprime to $b$ if and only if $\ell = sb +r$ where $s, r$ are integers satisfying $0 \leq s < a$, $1 \leq r < b$, and $\gcd(r,b) =1$. Hence, there are precisely $a\varphi(b)$ choices for~$\ell$.
\end{proof}%

\begin{lemma}\label{l:more:3}
	Let $G$ be a cyclic group and let $t$ be a divisor of~$\vert G\vert$. Then $G$ contains $\vert G \vert\varphi(t)/t$ elements~$g$ such that $\vert G \vert/\vert g \vert$ and $t$ are coprime.
\end{lemma}

\begin{proof}
	Let $h$ be a generator of~$G$, whence the elements of $G$ are given by $h^\ell$, $\ell \in \{1, \dots, \vert G \vert\}$. 
	Since $\vert G \vert/\vert h^\ell\vert = \gcd(\vert G \vert, \ell)$, we see that $\vert G \vert/\vert h^\ell\vert$ is coprime to $t$ if and only if $\gcd(\vert G \vert, \ell)$ is coprime to $t$, which (recalling that $t$ is a divisor of $\vert G \vert$) is the case if and only if $\gcd(\ell, t)=1$.
	
	Thus, the number of elements $g$ in~$G$ satisfying $\gcd\bigl( \vert G \vert/\vert g \vert, t\bigr) = 1$ equals the number of integers in $\ell \in \{1, \dots, \vert G \vert\}$ such that $\gcd(\ell, t) =1$. Then the assertion holds by Lemma~$\ref{l:more:2}$ (applied to $a = \vert G \vert/t$ and $b = t$).
\end{proof}

A positive integer~$n$ is said to be \emph{square-free} if for all prime divisors $s$ of~$n$, $s^2$ does not divide $n$. Observe that $1$ is square-free.

\begin{lemma}\label{l:more:4}
	Let $m,t \in \mathbb{N}$. Let
		\[J = \biggl\{n \in \mathbb{N} \mid n \text{ divides } m, \, \gcd\biggl( \frac{q^m-1}{q^{m/n}-1}, t \biggr) =1\biggr\}.\]
	Then the following hold.
	\begin{enumerate}[$(a)$]
		\item If $j \in J$ and $j_0 \mid j$, then $j_0 \in J$. 
		\item Let $r$ be the product of all distinct primes in~$J$ if any exist, and let $r = 1$ else. Then
	$\{n \in \mathbb{N} \mid n \text{ divides } r\} = \{j \in J \mid j \text{ square-free}\}$.
	\end{enumerate}\vskip-.3cm
\end{lemma}

\begin{proof}
	\begin{enumerate}[$(a)$]
		\item Let $j \in J$ and let $j_0$ be a divisor of~$j$. Seeking a contradiction, assume that $j_0 \notin J$. Then
		\[\gcd\left( \frac{q^m-1}{q^{m/j_0}-1}, t \right) \neq 1.\]
	Since $m/j$ divides $m/j_0$, by Lemma~$\ref{l:more:1}$ we have $q^{m/j}-1 \mid q^{m/j_0}-1$. Then $\gcd\bigl( (q^m-1)(q^{m/j_0}-1)^{-1}, t \bigr) \mid \gcd\bigl( (q^m-1)(q^{m/j}-1)^{-1}, t \bigr)$. In particular, $\gcd\bigl( (q^m-1)(q^{m/j}-1)^{-1}, t \bigr) \neq 1$. This is not true since $j\in J$.
	
		\item The assertion trivially holds for $J = \{1\}$. So suppose that $J \neq \{1\}$. Then, by part~$(a)$ of the current lemma, $J$ contains primes.
		Let $r_1, \dots, r_\ell$ be (all) the distinct primes in $J$, whence  $r = \prod_{i=1}^\ell r_i$. 	Since $1 \in J$, in order to prove the assertion it suffices to show that
		\[ \{n \in \mathbb{N} \mid n \geq 2, n \text{ divides } r\} = \{ j \in J \mid j \geq 2, j \text{ square-free} \}.\]
		Consider a divisor $n \geq 2$ of~$r$. We may assume that $n = \prod_{i=1}^k r_i$ for some $k \leq \ell$.
	Since $r_1, \dots, r_k$ are pairwise distinct prime divisors of~$m$, their product $n$ is a square-free divisor of~$m$.
	Let $(q^m-1)_t$ be the product of the $s$-parts of~$q^m-1$ for all prime divisors $s$ of~$t$. Observe that, for all $i \in \{1, \dots, k\}$, the condition $\gcd\bigl( (q^m-1)(q^{m/r_i}-1)^{-1}, t \bigr) =1$ implies that $(q^m-1)_t \mid \! q^{m/r_i}-1$. Thus, $(q^m-1)_t$ divides $\gcd\bigl( q^{m/r_1}-1, \dots, q^{m/r_k}-1 \bigr)$, which according to Lemma~$\ref{l:more:1}$ equals $q^{\gcd(m/r_1, \dots, m/r_k)} -1 = q^{m/n}-1$. Then $\gcd\bigl( (q^m-1) (q^{m/n}-1)^{-1}, t \bigr) = 1 $, whence $n \in J$.
	
	Conversely, consider a square-free element $j \geq 2$ of~$J$. By part~$(a)$ of the current lemma, each prime divisor of~$j$ lies in~$J$. Hence, $j\mid r$.\qedhere
	\end{enumerate}%
\end{proof}

\section{Almost \emph{t}-hyper-irreducible polynomials}\label{s:proofs}

Let $t \in \mathbb{N}$. Recall that a polynomial $f \in \mathbb{F}_q[x]$ is said to $t$-hyper-irreducible if $f(x^t)$ is irreducible. We refer to $f$ as almost $t$-hyper-irreducible if $f$ is irreducible and $f(x^t)$ contains an irreducible factor of degree strictly bigger than $\deg(f)t/2 = \deg(f(x^t))/2$. As we point out in the introduction, any $t$-hyper-irreducible polynomial is irreducible, which is why $t$-hyper-irreducible polynomials are almost $t$-hyper-irreducible.

This section is devoted to the occurrence of almost $t$-hyper-irreducible polynomials in~$\mathbb{F}_q[x]$. The proofs of Theorems~$\ref{t:existence:1}$, $\ref{t:number1:1}$, $\ref{t:number2:1}$ are given in Subsections~$\ref{ss:existence}$, $\ref{ss:number1}$, and~$\ref{ss:number2}$ respectively.

\subsection{Existence of almost \emph{t}-hyper-irreducible polynomials}\label{ss:existence}

\begin{proposition}\label{p:existence:2}
	Let $m,t,e$ be positive integers satisfying $e > mt/2$. Let $f \neq x$ be an irreducible polynomial of degree~$m$ in $\mathbb{F}_q[x]$ and let $\omega$ be a root of~$f$. Then $f(x^t)$ contains an irreducible (over~$\mathbb{F}_q$) factor of degree~$e$ if and only if $\gcd(q,t) = 1$ and $\ord{q}{\vert \omega \vert t} =e$.
\end{proposition}

\begin{proof}	
	First, assume that $f(x^t)$ has an irreducible factor $f_0 \in \mathbb{F}_q[x]$ with $\deg(f_0) =~e$.
	If $\gcd(q,t) \neq 1$ then the characteristic~$p$ of~$\mathbb{F}_q$ divides~$t$. In this case, writing $f(x) = \sum_{i=0}^m \alpha_i x^i$, we get
		\[f(x^t) = \sum_{i=0}^m \alpha_i x^{ti} = \left(\sum_{i=0}^m \alpha_i^{q/p} x^{ti/p}\right)^p,\]
	which yields the contradiction $e \leq mt/p \leq mt/2$. It follows that $t$ and $q$ are coprime. Let $\xi$ be a root of~$f_0$. Note that $\xi \neq 0$. By Lemma~$\ref{l:poly:3}$ we have
		\begin{align}
			\ord{q}{\vert \xi \vert} = e.\label{eq:existence:2:1}
		\end{align}
	Recalling that $f_0 \mid f(x^t)$ we see that $\xi^t$ is a root of $f$. Since $f$ is irreducible, by Lemma~$\ref{l:poly:1}(d)$ we have $\vert \xi^t \vert = \vert \omega \vert$. Then
		\begin{align}
			\vert \xi \vert = \vert \omega \vert \gcd(t, \vert \xi \vert).\label{eq:existence:2:2}
		\end{align}
	By Lemma~$\ref{l:ord:1}(f)$ we have
		\begin{align}
			\ord{q}{\vert \omega \vert \gcd(t, \vert \xi \vert)} \leq \ord{q}{\vert \omega\vert} \gcd(t,\vert \xi \vert).\label{eq:existence:2:3}
		\end{align}
	Now, by~$\eqref{eq:existence:2:1}$ and $\eqref{eq:existence:2:2}$ the left hand-side of $\eqref{eq:existence:2:3}$ is equal to $e > mt/2$. Further, by Lemma~$\ref{l:poly:3}$ the right hand-side of $\eqref{eq:existence:2:3}$ equals $m\gcd(t,\vert \xi \vert)$. This reveals that $\gcd(t,\vert \xi\vert) = t$.
	Then by $\eqref{eq:existence:2:1}$ and $\eqref{eq:existence:2:2}$ we get $\ord{q}{\vert \omega \vert t} = e$.
	
	Conversely, assume that $\gcd(q,t) = 1$ and $\ord{q}{\vert \omega \vert t} =e$. According to Lemma~$\ref{l:poly:6}$ the polynomial $f(x^t)$ contains a root $\xi$ of order~$\vert \omega \vert t$. Let $f_0$ be an irreducible (over~$\mathbb{F}_q$) factor of~$f(x^t)$ which contains $\xi$ as a root. Then by Lemma~$\ref{l:poly:3}$ we have $\deg(f_0) = \ord{q}{\vert \xi \vert} = \ord{q}{\vert \omega \vert t} = e$.
\end{proof}

We are now in the position to prove Theorem~$\ref{t:existence:1}$.

\begin{proof}[Proof of Theorem~$\ref{t:existence:1}$.]
	Let $f \in \mathbb{F}_q[x]$ be irreducible such that $\deg(f) = m$ and $f(x^t) \in \mathbb{F}_q[x]$ has an irreducible factor of degree~$e$. If $t = 1$ then $\gcd(t,q) = 1$ and $e = m = \ord{q}{(q^m-1)t}$. So suppose that $t \geq 2$. (If $0$ is a root of $f$, then recalling that $f$ is irreducible we have $f = x$. But, since $t \geq 2$, the irreducible factors of $x^t$ have degree~$1 \leq\deg(f)t/2$). Thus, $f(0) \neq 0$. Now, let $\omega$ be a root of~$f$. Then Proposition~$\ref{p:existence:2}$ reveals that
		\[\gcd(q,t) = 1, \quad e = \ord{q}{\vert \omega\vert t}.\]
	Since $\omega \in \mathbb{F}_{q^m}^*$ by Lemma~$\ref{l:poly:1}(c)$, $\vert \omega \vert t$ divides~$(q^m-1)t$, which is why by Lemma~$\ref{l:ord:1}(b)$ we have $\ord{q}{\vert \omega \vert t} \mid \ord{q}{(q^m-1)t}$, that is (recalling that $ e = \ord{q}{\vert \omega\vert t}$),
		\[ e \mid \ord{q}{(q^m-1)t}. \]
	Using Lemma~$\ref{l:ord:1}(f)$ (by which $\ord{q}{(q^m-1)t}\leq\ord{q}{q^m-1}t = mt$) and the assumption $e>mt/2$, it follows that $\ord{q}{(q^m-1)t}=e$.

	Conversely, suppose that $q,t$ are coprime and $\ord{q}{(q^m-1)t} = e$. Let~$\omega$ be a primitive element of~$\mathbb{F}_{q^m}^*$ (whence $\vert \omega\vert = q^m-1$) and let $f$ be the minimal polynomial of~$\omega$ over~$\mathbb{F}_q$ (whence $f \in \mathbb{F}_q[x]$ is irreducible with $\deg(f) = m$). By Proposition~$\ref{p:existence:2}$ the polynomial~$f(x^t)$ contains an irreducible (over~$\mathbb{F}_q$) factor of degree~$e$.%
\end{proof}

The following corollary is obtained by combining Propositions~$\ref{p:ord:3}$ with Proposition~$\ref{p:existence:2}$, and Theorem~$\ref{t:existence:1}$ respectively. 
The characterisation of $t$-hyper-irreducible polynomials presented in part~$(a)$ of Corollary~$\ref{c:existence:3}$ generalises Lemma~$\ref{l:poly:3}$ (and we can retrieve the statement of Lemma~$\ref{l:poly:3}$ by setting $t = 1$). It also generalises \cite[Theorem~$3.75$]{MR1429394} which covers the case $m = 1$.

\begin{corollary}\label{c:existence:3}
	\begin{enumerate}[$(a)$]
		\item Let $f\in \mathbb{F}_q[x]$ contain a non-zero root~$\omega$, let $t \in \mathbb{N}$, and let $m = \deg(f)$. The following are equivalent.
			\begin{enumerate}[$(i)$]
				\item The polynomial~$f$ is $t$-hyper-irreducible.
				\item The integers $t,q$ are coprime and $\ord{q}{\vert \omega \vert t} = m t$.
				\item We have $\ord{q}{\vert \omega \vert} = m$, the integer~$\gcd(4,t)$ divides~$\vert \omega \vert$, and each prime divisor of~$t$ divides~$\vert \omega \vert$ but not $(q^m-1)/\vert \omega \vert$.
			\end{enumerate}
				
		\item Let $m,t \in \mathbb{N}$. The following are equivalent.
	\begin{enumerate}[$(i)$]
		\item There exists a $t$-hyper-irreducible polynomial of degree~$m$ in~$\mathbb{F}_q[x]$.
		\item The integers $t,q$ are coprime and $\ord{q}{(q^m-1)t} = mt$.
		\item Writing $t_1, \dots, t_\ell$ for (all) the distinct odd prime divisors of~$t$, we have $\gcd(t,4) \prod_{i=1}^\ell t_i \mid (q^m-1)$.
	\end{enumerate}
	\end{enumerate}
\end{corollary}

\begin{proof}
	\begin{enumerate}[$(a)$]
		\item If condition~$(i)$ holds, that is if $f$ is $t$-hyper-irreducible, then $f$ is irreducible, and by Proposition~$\ref{p:existence:2}$ condition~$(ii)$ follows. By Proposition~$\ref{p:ord:3}$ condition~$(iii)$ implies $(ii)$.
		
		It remains to show that $(ii)$ entails both~$(i)$ and~$(iii)$. So suppose that $\gcd(q,t) = 1$ and $\ord{q}{\vert \omega \vert t} = mt$. Since (as we may deduce from Lemma~$\ref{l:poly:3}$) $\ord{q}{\vert \omega \vert} \leq \deg(f) = m$, using Lemmas~$\ref{l:ord:1}(f), \ref{l:poly:3}$ we see that $\ord{q}{\vert \omega \vert} = m$. Then by Proposition~$\ref{p:ord:3}$ (applied to $a = q$ and $r = \vert \omega \vert$) condition~$(iii)$ holds. Moreover, by Lemma~$\ref{l:poly:3}$ the polynomial~$f$ is irreducible, which combined with Proposition~$\ref{p:existence:2}$ shows that $(i)$ is satisfied.
	
	\item By Theorem~$\ref{t:existence:1}$ conditions~$(i)$ and~$(ii)$ are equivalent. The equivalence of conditions~$(ii)$ and~$(iii)$ holds by Proposition~$\ref{p:ord:3}$ (applied to $a = q$ and $r = q^m-1$). \qedhere
	\end{enumerate}
\end{proof}

We conclude this subsection verifying the existence of (almost) $t$-hyper-irreducible polynomials for some specific values for $q, m$, and $t$. 

\begin{example}\label{e:existence:4}
	\begin{enumerate}[$(a)$]
		\item Let $q = 5$, $m =5$, $t = 99$. Then $\gcd(t,q) = 1$ and (as we may calculate by hand or in \textsf{GAP}\cite{GAP4} by calling \texttt{OrderMod(5,(5\^{}5-1)*99)} we have $\ord{5}{(5^5-1)99} = 330 > mt/2$.
	By Theorem~$\ref{t:existence:1}$ there are irreducible polynomials~$f \in \mathbb{F}_5[x]$ of degree~$5$ such that $f(x^{99})$ has an irreducible (over~$\mathbb{F}_5$) factor of degree~$330$. 
	
		\item According to Corollary~$\ref{c:existence:3}(b)$ there exists a $100$-hyper-irreducible polynomial of degree~$10$ over~$\mathbb{F}_q$ if and only if $20 \mid q^{10}-1$. By Euler's totient theorem we have $20 \mid q^{\varphi(20)}-1 = q^8-1$. Hence, by Lemma~$\ref{l:more:1}$, $20\mid q^{10}-1$ if and only if $20$ divides $q^{\gcd(8,10)}-1 = (q+1)(q-1)$. We conclude that $\mathbb{F}_q[x]$ contains $100$-hyper-irreducible polynomials of degree~$10$ if and only if $q \equiv \pm 1 \mod{10}$.
		
		\item Suppose that $q$ is odd. Then $2 \mid q^m-1$ for all $m\in \mathbb{N}$ and, by Corollary~$\ref{c:existence:3}(b)$, the polynomial ring $\mathbb{F}_q[x]$ contains $2$-hyper-irreducible polynomials of any (positive) degree.
	\end{enumerate}
\end{example}

\subsection{Counting monic \emph{t}-hyper-irreducible polynomials}\label{ss:number1}

Recall (from Definition~$\ref{d:num t-hyper-irred}$) that $N_q^*(m,t)$ is the number of all monic $t$-hyper-irreducible polynomials~$f \neq x$ of degree~$m$ in~$\mathbb{F}_q[x]$. Recall further the definition of the Moebius function~$\mu:\mathbb{N} \rightarrow \{-1,0,1\}$ given in~$\eqref{eq:moebius}$. 
 
\begin{proof}[Proof of Theorem~$\ref{t:number1:1}(a)$]
The proof follows the approach taken in~\cite{chebolu} to derive a formula for the number of all monic, irreducible polynomials of a given degree over~$\mathbb{F}_q$.

	Suppose that $N_q^*(m,t) \neq 0$.
	Let $\widehat{f} \in \mathbb{F}_q[x]$ be the product of all monic, $t$-hyper-irreducible polynomials~$f\neq x$ of degree~$m$ over~$\mathbb{F}_q$, and let~$R$ be the set of all roots of~$\widehat{f}$. (Recall that $t$-hyper-irreducible polynomials are irreducible.) By Lemma~$\ref{l:poly:1}(a)(e)$ each irreducible factor of~$\widehat{f}$ has $m$ distinct roots, and no two distinct irreducible factors of~$\widehat{f}$ share a root. Hence,
		\begin{align}
			N_q^*(m,t) =  \frac{\vert R \vert}{m}. \label{eq:number1:1:1}
		\end{align}
	Let $t_0$ be the product of~$\gcd(4,t)$ and all distinct odd prime divisors of~$t$. Since $t$ and $t_0$ have the same prime divisors (possibly with different multiplicities), using the product formula for Euler's totient function we obtain
		\begin{align}
			\frac{\varphi(t_0)}{t_0} = \frac{\varphi(t)}{t}.\label{eq:number1:1:2}
		\end{align}
	By Lemma~$\ref{l:poly:1}(d)$, $R \subseteq \mathbb{F}_{q^m}^*$. Then Corollary~$\ref{c:existence:3}(a)$ yields
		\[R = \biggl\{ \omega \in \mathbb{F}_{q^m}^* \mid \ord{q}{\vert \omega \vert}=m, \;t_0 \text{ divides } \vert \omega \vert, \; \gcd\biggl( \frac{q^m-1}{\vert\omega\vert}, t_0 \biggr) =1 \biggr\}. \]
	Since $N_q^*(m,t) \neq 0$, by Corollary~$\ref{c:existence:3}(b)$ the integer $t_0$ divides $q^m-1$. This shows that any element $\omega \in \mathbb{F}_{q^m}^*$ satisfying $\gcd\bigl( (q^m-1)\vert \omega \vert^{-1}, t_0 \bigr) = 1$ has order divisible by $t_0$. Hence, 
		\[R = \biggl\{ \omega \in \mathbb{F}_{q^m}^* \mid \ord{q}{\vert \omega \vert} = m,  \;\gcd\biggl(\frac{q^m-1}{\vert\omega\vert}, t_0 \biggr) =1 \biggr\}.\]
	By Lemma~$\ref{l:poly:2}$, for $\omega \in \mathbb{F}_{q^m}^*$ the condition $\ord{q}{\vert \omega \vert} = m$ is equivalent to saying that $\omega$ does not lie in any maximal subfield of~$\mathbb{F}_{q^m}$ containing~$\mathbb{F}_q$. Such maximal subfields have order $q^{m/j}$ where $j$ is a prime dividing~$m$.~Thus,
	\begin{align}
		R = \biggl\{ \omega \in \mathbb{F}_{q^m}^* \mid \gcd\biggl( \frac{q^m-1}{\vert\omega\vert}, t_0 \biggr) =1 \biggr\} \setminus A,\label{eq:number1:1:3}
	\end{align}
	where
		\[A = \bigcup_{\substack{j \mid m, \\ j \text{ prime}}} \hspace*{-.1cm} \biggl\{ \omega \in \mathbb{F}_{q^{m/j}}^* \mid \gcd\biggl( \frac{q^m-1}{\vert\omega\vert}, t_0 \biggr) =1 \biggr\}.  \]
	If $j$ is a prime divisor of $m$ which is not an element of~$J$ (as defined in the assumption), then $\gcd\bigl( (q^m-1)\vert \omega \vert^{-1}, t_0 \bigr) \neq 1$ for all $\omega \in \mathbb{F}_{q^{m/j}}^*$, in which case the set $\bigl\{ \omega \in \mathbb{F}_{q^{m/j}}^* \mid \gcd\bigl( (q^m-1)\vert\omega\vert^{-1}, t_0 \bigr) =1 \bigr\}$ is empty. This shows that
		\[A = \bigcup_{\substack{j \in J, \\ j \text{ prime}}} \hspace*{-.1cm} \biggl\{ \omega \in \mathbb{F}_{q^{m/j}}^* \mid \gcd\biggl( \frac{q^m-1}{\vert\omega\vert}, t_0 \biggr) =1 \biggr\}.  \]
	If $j_1, \dots, j_\ell \in J$ are distinct primes and $s=\prod_{i=1}^\ell j_i$ then the intersection $\bigcap_{i=1}^\ell \mathbb{F}_{q^{m/j_i}}^*$ is equal to $\mathbb{F}_{q^{m/s}}^*$. Using the inclusion-exclusion principle, it follows that
		\begin{align*}
			\vert A \vert
				& = \sum_{\substack{j \in J,\\ j\text{ prime}}} \hspace*{-.1cm}  \biggl\vert \biggl\{\omega \in \mathbb{F}_{q^{m/{j}}}^* \mid \gcd\biggl( \frac{q^m-1}{\vert\omega\vert}, t_0 \biggr) =1 \biggr\} \biggr\vert \\
				& \quad - \hspace*{-.5cm} \sum_{\substack{j_1,j_2 \in J,\\ j_1,j_2\text{ distinct primes}}} \hspace*{-.5cm}  \biggl\vert \biggl\{\omega \in \mathbb{F}_{q^{m/(j_1j_2)}}^* \mid \gcd\biggl( \frac{q^m-1}{\vert\omega\vert}, t_0 \biggr) =1 \biggr\} \biggr\vert \\
				& \qquad + \hspace*{-.5cm}  \sum_{\substack{j_1,j_2,j_3 \in J,\\ j_1,j_2,j_3\text{ distinct primes}}} \hspace*{-.5cm} \biggl\vert \biggl\{\omega \in \mathbb{F}_{q^{m/(j_1j_2j_3)}}^* \mid \gcd\biggl( \frac{q^m-1}{\vert\omega\vert}, t_0 \biggr) =1 \biggr\} \biggr\vert \\
				& \qquad \quad - \dots
		\end{align*}
	As we may deduce from Lemma~$\ref{l:more:4}(b)$, a product of distinct primes from~$J$ is a square-free element of~$J$; and moreover, each non-trivial, square-free element of~$J$ is a product of distinct primes from~$J$. Thus,
	\[ \vert A \vert = -\sum_{\substack{j \in J,\\ j \neq 1}} \mu(j) \biggl\vert \biggl\{\omega \in \mathbb{F}_{q^{m/j}}^* \mid \gcd\biggl( \frac{q^m-1}{\vert\omega\vert}, t_0 \biggr) =1 \biggr\} \biggr\vert.\]
	Now, consider an element $j \in J$. Then (by definition) we have
		\begin{align}
			\gcd\biggl( \frac{q^m-1}{q^{m/j}-1},t_0 \biggr) = 1.\label{eq:number1:1:4}
		\end{align}
	Recalling that $t_0 \mid q^m-1$, $\eqref{eq:number1:1:4}$ implies that $t_0 \mid q^{m/j}-1$. Condition $\eqref{eq:number1:1:4}$ also implies that an element $\omega \in \mathbb{F}_{q^{m/j}}^*$ satisfies $\gcd\bigl( (q^m-1)\vert \omega \vert^{-1}, t_0\bigr) =1$ if and only if $\gcd\bigl( (q^{m/j}-1)\vert \omega \vert^{-1}, t_0\bigr) =1$. Hence, 
	\[\vert A \vert = -\sum_{\substack{j \in J,\\ j \neq 1}} \mu(j) \biggl\vert \biggl\{\omega \in \mathbb{F}_{q^{m/j}}^* \mid \gcd\biggl( \frac{q^{m/j}-1}{\vert\omega\vert}, t_0 \biggr) =1 \biggr\} \biggr\vert,\]
	and thus by $\eqref{eq:number1:1:3}$,
		\[\vert R \vert = \sum_{j \in J} \mu(j) \biggl\vert \biggl\{\omega \in \mathbb{F}_{q^{m/j}}^* \mid \gcd\biggl( \frac{q^{m/j}-1}{\vert\omega\vert}, t_0 \biggr) =1 \biggr\} \biggr\vert.\]
	Then Lemma~$\ref{l:more:3}$ (applied to~$G = \mathbb{F}_{q^{m/j}}^*$ and $t = t_0$) yields 
		\[\vert R \vert = \frac{\varphi(t_0)}{t_0} \sum_{j \in J} \mu(j) (q^{m/j}-1),\]
	which combined with Equations~$\eqref{eq:number1:1:1}$, $\eqref{eq:number1:1:2}$ finalises the proof. 
\end{proof}

\begin{proof}[Proof of Theorem~$\ref{t:number1:1}(b)$] 
	Let $r$ be as defined in Lemma~$\ref{l:more:4}(b)$. (Recall that $\mu(n) = 0$ if $n\in \mathbb{N}$ is not square-free.)
	Then according to Lemma~$\ref{l:more:4}(b)$ and Theorem~$\ref{t:number1:1}(a)$ we obtain
		\[N_q^*(m,t) = \frac{\varphi(t)}{mt} \sum_{n\mid r} \mu(n) (q^{m/n}-1),\]
	and thus by $\eqref{eq:N_q(m)}$,
		\begin{align}
			N_q^*(m,t) = \frac{\varphi(t)}{mt} r N_{q^{m/r}}^*(r).\label{eq:number1:1:5}
		\end{align}
	According to Lemma~$\ref{l:poly:5}(b)$ we get $r N_{q^{m/r}}^*(r) \leq (q^{m/r})^r-1 = q^m-1$, which (together with~$\eqref{eq:number1:1:5}$) proves the upper bound for $N_q^*(m,t)$. Further, from Lemma~$\ref{l:poly:5}(a)$ it follows that $r N_{q^{m/r}}^*(r) \geq m N_q^*(m)$. By Lemma~$\ref{l:poly:5}(b)$ we then obtain $r N_{q^{m/r}}^*(r) \geq m(q^m-1)/(m+1)$, which (combined with $\eqref{eq:number1:1:5}$) verifies the lower bound for~$N_q^*(m,t)$.%
\end{proof}

Comparing Lemma~$\ref{l:poly:5}(b)$ with Theorem~$\ref{t:number1:1}(b)$ we see the following. If~$\mathbb{F}_q[x]$ contains monic $t$-hyper-irreducible polynomials of degree~$m$ (for some $m,t\in \mathbb{N}$), then the number of all such polynomials is roughly equal to $\varphi(t)/t$ times the number of all monic irreducible polynomials of degree~$m$ over~$\mathbb{F}_q$.

\subsection{Counting monic almost \emph{t}-hyper-irreducible polynomials}\label{ss:number2}

For $m,t\in \mathbb{N}$, recall the Definition~$\ref{d:num t-hyper-irred}$ of $N_q^*(m,t)$. A formula for $N_q^*(m,t)$ is given in Theorem~$\ref{t:number1:1}(a)$. Further, for a positive integer~$t$ and a prime~$s$, recall the meaning of~$(t)_s$ and $(t)_{s'}$.

\begin{proof}[Proof of Theorem~$\ref{t:number2:1}$]
	Let $T$ be the number of all monic, irreducible polynomials~$f \neq x$ in~$\mathbb{F}_q[x]$ such that $\deg(f) = m$ and $f(x^t)$ contains an irreducible factor of degree~$e$. According to the assumption, we have
			\[ \gcd(t, q) = 1 \quad \text{and} \quad \ord{q}{(q^m-1)t} =e. \]
	Since $t=1$ yields the contradiction $e = \ord{q}{q^m-1} = m = mt$, it follows that $t \geq 2$. Moreover, by Theorem~$\ref{t:existence:1}$ we get $T \neq 0$.

	By Lemma~$\ref{l:ord:1}(b)$, $\ord{q}{q^m-1}$ divides~$\ord{q}{(q^m-1)t}$, that is $m\mid e$.
	Further, by Lemma~$\ref{l:ord:4}$ (applied to $a = q$, $r = q^m-1$),
		\[ s \text{ is an odd prime}, \quad s \nmid q^m-1.\]
	Let $\mathfrak{T}$ be the set of all monic, irreducible polynomials~$f\in \mathbb{F}_q[x]$ such that $\deg(f) = m$ and $f(x^t)$ contains an irreducible factor of degree~$e$. Observe that $x \notin \mathfrak{T}$. (This is because, for $t \geq 2$, the irreducible factors of $x^t$ have degree~$1 \leq mt/2$.) By definition,
		\[T = \vert \mathfrak{T}\vert.\]
	We prove the assertion by showing that $\mathfrak{T}$ is the set of all monic, $(t)_{s'}$-hyper-irreducible polynomials~$f \neq x$ of degree~$m$ over~$\mathbb{F}_q$.
	
	To this end, consider a polynomial $f \in \mathfrak{T}$ and let $\omega$ be a root of~$f$. By Lemma~$\ref{l:poly:1}(d)$, $\omega$ lies in~$\mathbb{F}_{q^m}^*$, whence $\vert \omega \vert$ divides $q^m-1$. In particular, $\gcd(q,\vert\omega\vert)=1$. Thus (recalling that $t$ and $q$ are coprime) we get $\gcd(q,\vert \omega \vert t) = 1$. Further by Lemma~$\ref{l:poly:3}$ and Proposition~$\ref{p:existence:2}$ we see that $\ord{q}{\vert \omega \vert} \!= m$ and $\ord{q}{\vert \omega \vert t} \!= e$. Then according to Lemma~$\ref{l:ord:4}$ (applied to $a=q$ and $r = \vert \omega \vert$) we obtain $\ord{q}{\vert \omega \vert (t)_{s'}} = m(t)_{s'}$. By Corollary~$\ref{c:existence:3}(a)$ this means that $f$ is $(t)_{s'}$-hyper-irreducible.
	
	Conversely, let $f\neq x$ be a monic, $(t)_{s'}$-hyper-irreducible polynomial of degree~$m$ over~$\mathbb{F}_q$. Let $\omega$ be a root of~$f$. Again, by Lemma~$\ref{l:poly:1}(d)$ the order of $\omega$ divides $q^m-1$. Recalling that $s \nmid q^m-1$, it follows that $s \nmid \vert \omega \vert$, and thus by Lemma~$\ref{l:ord:1}(c)$,
		\[\ord{q}{\vert \omega \vert t} = \lcm\bigl\{ \ord{q}{\vert \omega \vert (t)_{s'}}, \ord{q}{(t)_s} \bigr\}.\]
	As we may deduce from Corollary~$\ref{c:existence:3}(a)(b)$, we have
		\[\ord{q}{\vert \omega \vert (t)_{s'}} = \ord{q}{(q^m-1) (t)_{s'}}.\] Hence, 
		\[\ord{q}{\vert \omega \vert t} = \lcm\bigl\{ \ord{q}{(q^m-1) (t)_{s'}}, \ord{q}{(t)_s} \bigr\},\]
	which (recalling that $s \nmid q^m-1$) by Lemma~$\ref{l:ord:1}(c)$ simplifies to
		\[\ord{q}{\vert \omega \vert t} = \ord{q}{(q^m-1)t}.\]
	Thus, $\ord{q}{\vert \omega \vert t} = e$, and then $f\in \mathfrak{T}$ by Proposition~$\ref{p:existence:2}$.
\end{proof}

In the following example, we apply Theorems~$\ref{t:number1:1}(a), \ref{t:number2:1}$ in order to determine the number of certain almost $99$-hyper-irreducible polynomials over~$\mathbb{F}_5$.

\begin{example}\label{e:number2:2}
	Let $q = 5$, $m =5$, and $t = 99$. According to Example~$\ref{e:existence:4}(a)$ there exists a monic, irreducible polynomial~$f \in \mathbb{F}_5[x]$ such that $\deg(f) = 5$ and $f(x^{99})$ has an irreducible (over~$\mathbb{F}_5$) factor of degree~$330$.\\
	Since $99/\gcd(330/5,99) = 3$, by Theorem~$\ref{t:number2:1}$ the number of all such polynomials is given by $N_5^*(5, (99)_{3'}) = N_5^*(5,11)$, which according to Theorem~$\ref{t:number1:1}(a)$ is equal to $\varphi(11) (5^5-1)/55 = 568$.
	
	For comparison: By $\eqref{eq:N_q(m)}$ the number of all monic, irreducible polynomials of degree~$5$ in~$\mathbb{F}_5[x]$ equals $N_5^*(5) = (5^5 - 5^1)/5 = 624$.
\end{example}

\section*{Acknowledgements}

The results presented here have been obtained during the author's PhD candidature as a cotutelle student at the University of Western Australia and the RWTH Aachen University, partly supported by a scholarship awarded by the Studienstiftung des deutschen Volkes (German Academic Scholarship Foundation) and the Australian Research Council Discovery Project DP190100450.
 The author thanks her supervisors Cheryl Praeger and Gerhard Hi\ss{} for valuable discussions, careful reading of her work, and many extremely helpful comments.

\end{document}